\documentclass[11pt, reqno, a4paper]{amsart} 
\usepackage[english]{babel}
\usepackage{amsfonts, amsmath, amsthm, amssymb,amscd,indentfirst}
\usepackage{times}
\usepackage{enumerate}

\usepackage{anysize} 
\marginsize{1.5cm}{1.5cm}{1.5cm}{1.5cm}

\newtheorem{theorem}{Theorem}[section]
\newtheorem{proposition}[theorem]{Proposition}
\newtheorem{lemma}[theorem]{Lemma}
\newtheorem{definition}[theorem]{Definition}

\newcommand{\loc}{\operatorname{loc}}
\newcommand{\dist}{d}
\newcommand{\dx}{\textnormal{d}x}
\newcommand{\ds}{\textnormal{d}\sigma}
\newcommand{\esup}{\operatorname{ess} \operatorname{sup}}
\newcommand{\einf}{\operatorname{ess} \operatorname{inf}}
\newcommand{\spt}{\operatorname{supp}}

\title{On removable singularities for solutions of Neumann problem for elliptic equations involving variable exponent}
\author{Juan Alcon Apaza}

\address{Universidade Federal Fluminense, Instituto de Matemática, Campus do Gragoatá, Rua Prof. Marcos Waldemar de Freitas, s/n, bloco H, Niterói, RJ 24210-201, Brazil }
\email{jpablo@id.uff.br}

\begin{document}

\maketitle

\begin{abstract}
We study the removability of a singular set in the boundary of Neumann problem for elliptic equations with variable exponent. We consider the case where the singular set is compact, and give sufficient conditions for removability of this singularity for  equations in the variable exponent Sobolev space $W^{1,p(\cdot)}(\Omega)$.
\end{abstract}

\let\thefootnote\relax
\let\thefootnote\relax

\markright{ON REMOVABLE OF SINGULARITIES FOR SOLUTIONS OF NEUMANN PROBLEM}

\section{Introduction}

This paper is devoted to the study of  conditions guaranteeing  the removability of singular set for solutions of nonlinear elliptic equations with Neumann boundary conditions of the form:
\begin{equation} \label{1}
\left\{ \begin{aligned}
-\operatorname{div} A \left(x,u,\nabla u\right) + a \left(x,u\right)  + g(x,u) &= 0 & & \text { in } \Omega ,\\
A \left(x,u,\nabla u\right) \cdot \nu + b(x,u) + h(x,u)  &= 0 & & \text { on } \partial \Omega.
\end{aligned}
\right.
\end{equation}

\noindent Throughout the whole article  $\Omega$ is a bounded open set in $\mathbb{R}^{n}$, $n\geq 2$, with Lipschitz boundary $\partial \Omega$, and   $\Gamma \subset \partial \Omega$ is  a compact set. We always equip $\partial \Omega$ with the $(n-1)$-dimensional Hausdorff measure.

\noindent We assume that  $A:\Omega \times \mathbb{R} \times \mathbb{R}^n \rightarrow \mathbb{R}^n$, $a :\Omega \times \mathbb{R} \rightarrow \mathbb{R}$, $g:\bar{\Omega} \times \mathbb{R} \rightarrow \mathbb{R}$, $b, h :\partial \Omega \times \mathbb{R} \rightarrow \mathbb{R}$ are measurable functions, $g$ and  $h$ are locally bounded. Furthermore, there exists $\mu >0$ such that the following conditions are satisfied  almost everywhere:
\begin{gather}
 \left\langle A (x, u, \eta) , \eta \right\rangle \geq \mu\vert    \eta\vert    ^{p_1 (x)}, \label{4}\\
\left\vert    A(x, u, \eta)\right\vert     \leq \mu^{-1}\left(\vert    \eta\vert    ^{p_1 (x)-1}+\vert    u\vert    ^{p_1 (x)-1}+1 \right), \label{5}\\ 
\left\vert    a(x, u)\right\vert     \leq \mu^{-1}\left(\vert    u\vert    ^{p_1 (x)-1}+1 \right), \label{6}\\
g(x, u) \operatorname{sign} u \geq \mu \vert    u\vert    ^{p_2 (x)}-\mu ^{-1}, \label{7}\\
\left\vert    b(x, u)\right\vert     \leq \mu^{-1}\left(\vert    u\vert    ^{q_1 (x) -1}+1\right),\label{8}\\
h(x, u) \operatorname{sign} u \geq \mu \vert    u\vert    ^{q_2 (x)}-\mu^{-1}, \label{9}\\
A (x, u, -\eta)= -A (x, u, \eta), \label{23}
\end{gather}
where $p_1, p_2 : \Omega \rightarrow \mathbb{R}$, $q_1, q_2 : \partial \Omega \rightarrow \mathbb{R}$ are measurable functions satisfying 
\begin{gather}
p^-_1 , p_1 ^+ \in (1,n), \ q_1 ^- >1 , \quad \min\left\{p^- _2 , q^- _2\right\} - \max \left\{ p^+ _1 , q^+ _1\right\} +1 >0.\label{24}
\end{gather}
Also, for some $d\geq 0$,
\begin{gather}
\esup _{\Omega}\frac{p_1 p_2}{p_2 - p_1 +1} < n-d, \label{30}
\end{gather}
where $p^-_i = \einf _{\Omega} p_i $,   $p^+_i = \esup _{\Omega} p_i $,   $q^-_i = \einf _{\partial \Omega  } q_i $  and    $q^+_i  =   \esup _{\partial \Omega } q_i $; $i=1,2$.

Concerning the singular set $\Gamma \subset \partial \Omega$, we suppose that $\vert \Gamma \vert= 0$, and  there is a small enough $r_0 \in (0,1)$  so that  $\mathcal{U}=\{x\in \Omega \:\vert    \: \operatorname{dist} (x , \Gamma) < 2 r_0 \}$ is an open set with Lipschitz boundary $\partial \mathcal{U}$, $\{x\in \Omega \:\vert    \: \operatorname{dist}(x , \Gamma) = 2 r_0 \} \neq \emptyset$, and for a suitable positive constant $C>0$:  
\begin{gather} \label{51}
\left\vert  \left\{ \left\vert\operatorname{dist} \left( \cdot , \Gamma \right) - \ell \right\vert    < r \right\} \right\vert \leq C r^{n-d},  \left\vert  \left\{ \operatorname{dist} \left( \cdot , \Gamma \right)= r \right\} \right\vert \leq C r^{n-d-1}  \text { and }   \left\vert \partial \Omega \cap \left\{ \left\vert\operatorname{dist} \left( \cdot , \Gamma \right) - \ell \right\vert    < r \right\} \right\vert \leq C r^{n-d-1},
\end{gather}
where $0<r\leq\ell<r_0$.

The main result of this paper is the following theorem.

\begin{theorem} \label{34}
Suppose that the conditions \eqref{4} - \eqref{51}  are satisfied.  Let $u \in W^{1, p_1(\cdot)} _{\loc} \left(\bar{\Omega}\backslash \Gamma \right)$ $\cap L^\infty _{\loc}\left(\bar{\Omega}\backslash \Gamma\right)$ be a solution of equation \eqref{1} in $\bar{\Omega} \backslash \Gamma$. Then, the singularity of $u$ at  $\Gamma$ is removable.
\end{theorem}

We follow the same lines as in  \cite{fu} and \cite{mam}; which handle the case when the singular set is an interior point of $\Omega$, and the solution of an elliptic equation $u$ has no boundary conditions.

\section{Preliminaries}

We first  recall some facts on spaces $L^{p(\cdot)} (\Omega)$ and $W^{1,p(\cdot)} (\Omega)$. Denote by $\mathbf{P} \left( \Omega \right)$  the set of all Lebesgue measurable functions $p:\Omega \rightarrow [1, \infty]$. For the details see \cite{fan, kova, sam, dien}. Let $p\in \mathbf{P} \left( \Omega \right)$, we define the functional 
$$
\rho _{p}(u) = \int _{\Omega \backslash \Omega _{\infty}} \vert    u\vert    ^{p} \dx + \esup _{\Omega _{\infty}} \vert    u\vert    ,
$$
where $\Omega _{\infty} = \{x\in \Omega \:\vert    \: p(x) =\infty\}$.

The variable exponent Lebesgue space $L^{p(\cdot)}(\Omega)$ is the class of all functions $u$ such that $\rho_{p}(t u)<\infty$, for some $t>0$. $L^{p(\cdot)}(\Omega)$ is a Banach space equipped with the norm
$$
\Vert      u\Vert      _{L^{p(\cdot)} (\Omega)}=\inf \left\{\lambda>0\:\vert    \: \rho_{p}\left(\frac{u}{\lambda}\right) \leq 1\right\};
$$
see \cite[Theorem 2.5]{kova}.

\begin{proposition}\label{31} (see \cite[Theorem 2.1]{kova}.)
Let $p \in \mathbf{P}(\Omega)$. If $u \in L^{p(\cdot)}(\Omega)$ and $v \in L^{p^{\prime}(\cdot)}(\Omega)$, then
$$
\int_{\Omega}\vert    u v\vert     \dx \leq 2\Vert      u\Vert      _{L^{p(\cdot)} (\Omega)}\Vert      v\Vert      _{L^{p^\prime(\cdot)} (\Omega)} ,
$$
where 
$$
 p^{\prime}(x)= \left\{\begin{aligned}& \infty & & \text { if }  p(x)=1,\\ 
&1 & & \text { if } p(x) = \infty, \\ 
&\frac{p(x)}{p(x)-1} & & \text { if } p(x) \neq 1 \text { and }  p(x)\neq \infty.
\end{aligned} \right.
$$

\end{proposition}
\begin{proposition}\label{32}  (see \cite[Theorem 1.3]{fan}.)
Let $p \in \mathbf{P}(\Omega)$ with $p^{+}<\infty$. For any $u \in L^{p(\cdot)}(\Omega)$, we have
\begin{enumerate}[(1)]
\item if $\Vert      u\Vert      _{L^{p(\cdot)} (\Omega)} \geq 1$, then $\Vert      u\Vert      _{L^{p(\cdot)} (\Omega)}^{p^{-}} \leq \int_{\Omega}\vert    u\vert    ^{p} \dx \leq\Vert      u\Vert      _{L^{p(\cdot)} (\Omega)}^{p^{+}}$,
\item if $\Vert      u\Vert      _{L^{p(\cdot)} (\Omega)}<1$, then $\Vert      u\Vert      _{L^{p(\cdot)} (\Omega)}^{p^{+}} \leq \int_{\Omega}\vert    u\vert    ^{p} \dx \leq\Vert      u\Vert      _{L^{p(\cdot)} (\Omega)}^{p^{-}}$.
\end{enumerate}
\end{proposition} 

The variable exponent Sobolev space $W^{1, p(\cdot)}(\Omega)$ is the class of all functions $u \in L^{p(\cdot)}(\Omega)$ which have the property $\vert    \nabla u\vert     \in L^{p(\cdot)}(\Omega)$. The space $W^{1, p(\cdot)}(\Omega)$ is a Banach space equipped with the norm
$$
\Vert      u\Vert      _{W^{1, p(\cdot)} (\Omega)}=\Vert      u\Vert      _{L^{p(\cdot)}(\Omega)}+\Vert      \nabla u\Vert      _{L^{p(\cdot)}(\Omega)} .
$$
More precisely, we have
\begin{proposition}(see \cite[Theorem 3.1]{kova}.) Let $p \in \mathbf{P}\left(\Omega\right)$. The space $W^{1,p(\cdot)} (\Omega)$ is a Banach space, which is separable if $p\in L^{\infty} \left(\Omega\right)$ and reflexive if $1< p^- \leq p^+  <\infty$.
\end{proposition}

Next we will see the definitions that we use in this work. Firstly we will make some observations regarding to the trace.  Let $p\in \mathbf{P}\left(\Omega\right)$. Obviously $W^{1,p(\cdot)}\left(\Omega\right) \subset W^{1,1}\left(\Omega\right)$ because $ p^- \geq 1$. From $W^{1,1} \left( \Omega \right) \rightarrow L^1 \left( \partial \Omega \right)$ we know that for all $u\in W^{1,p(\cdot)}\left(\Omega\right)$ there already holds $u\vert    _{\partial \Omega} \in L^1 \left(\partial \Omega \right)$. Thus for $W^{1,p(\cdot)} \left( \Omega \right)$, the trace $u\vert    _{\partial \Omega}$ has definite meaning; see \cite[p. 1398]{fanx}.

Define
\begin{equation*}
L^{p(\cdot)} _{\loc} \left( \bar{\Omega} \backslash \Gamma \right):= \{u:\Omega \rightarrow \mathbb{R}  \:\vert    \:  u\in L^{p(\cdot)}  \left(U\right) 
 \text { for all open subset } U\subset \Omega \text { with } \bar{U} \cap \Gamma = \emptyset \}
\end{equation*}
and
\begin{equation*}
W^{1,p(\cdot)} _{\loc} \left( \bar{\Omega} \backslash \Gamma \right):= \{u:\Omega \rightarrow \mathbb{R}  \:\vert    \:  u\in W^{1,p(\cdot)}  \left(U\right) 
\text { for all open subset } U\subset \Omega \text { with } \bar{U} \cap \Gamma = \emptyset \}.
\end{equation*}
Similarly we define $L^{\infty} _{\loc} \left( \bar{\Omega} \backslash \Gamma \right)$.

Let  $v : \Omega \rightarrow \mathbb{R}$ be a function, we call $\spt v =  \overline{\{x\in \Omega \:\vert    \: v(x) \neq 0\}}$ the \textit{support} of $v$. For $E\subset \mathbb{R}^n$ and $x\in \mathbb{R}^n$, we denote by $\dist (x,E)$ the Euclidean distance from $x$ to $E$.

\begin{definition}
We will say that $u \in W^{1, p_1(\cdot)} _{\loc} \left(\bar{\Omega}\backslash \Gamma \right)\cap L_{\loc}^{\infty}\left(\bar{\Omega} \backslash \Gamma \right)$ is a (weak) solution  of equation \eqref{1} in $\bar{\Omega} \backslash \Gamma$ if  
\begin{equation} \label{2}
\begin{split}
 \int_{\Omega}\left\langle A(\cdot , u , \nabla u) , \nabla \varphi \right\rangle + a( \cdot , u ) \varphi +  g(\cdot, u) \varphi  \dx  
+ \int _{\partial \Omega } b(\cdot , u) \varphi + h(\cdot , u) \varphi \ds = 0
\end{split}
\end{equation}
for all $\varphi \in W_{\loc}^{1, p_1(\cdot)}\left(\bar{\Omega} \backslash \Gamma \right) \cap L_{\loc}^{\infty}\left(\bar{\Omega} \backslash \Gamma \right)$, with   $\spt \varphi \subset \bar{\Omega} \backslash\Gamma$.
\end{definition}
\noindent Let us observe that the trace  of $u\in W^{1,p_1 (\cdot)} _{\loc} (\bar{\Omega} \backslash \Gamma) \cap L_{\loc}^{\infty}\left(\bar{\Omega} \backslash \Gamma \right)$ possibly is not defined on $\partial \Omega$, however it is defined on $\{x\in \partial \Omega \:\vert    \: \dist (x , \Gamma) >r \}$, and is essentially bounded, for small enough values $r>0$.
\begin{definition}
We will say that the solution $u$ of equation \eqref{1} in $\bar{\Omega} \backslash \Gamma$ has a removable singularity at  $\Gamma$: if $u \in W^{1, p_1(\cdot)} _{\loc} \left(\bar{\Omega}\backslash \Gamma \right)$ $\cap L^{\infty} _{\loc} \left(\bar{\Omega}\backslash \Gamma \right)$ implies $u \in W^{1, p_1(\cdot)}  \left(\Omega \right)$ $\cap L^{\infty}  \left(\Omega \right)$ and the equality \eqref{2} is fulfilled for all $\varphi \in W^{1, p_1(\cdot)}(\Omega) \cap L^{\infty}(\Omega)$.
\end{definition}

\section{The behavior of solutions near the  singular set}
The main result of this section is Theorem \ref{21}, which  will be used in the proof of Theorem \ref{34}. We begin with the following results.

\begin{lemma} \label{18} (see \cite[p. 1004]{fab}).
Let $0<\theta<1, \sigma>0, \xi(h)$ be a nonnegative function on the interval $\left[1 / 2 , 1\right]$, and let
$$
\xi(k) \leq C_0 (h-k)^{-\sigma}(\xi(h))^{\theta}, \quad 1 / 2 \leq k<h \leq 1 ,
$$
for some positive constant $C_0$. Then, there exists $C_1(\sigma, \theta)>0$ such that
$$
\xi\left(1 / 2 \right) \leq C_1 C_0 ^{\frac{1}{1-\theta}}.
$$
\end{lemma}

\begin{lemma}  If $p\in (1,n)$, $u \in W^{1, p}(\mathcal{U})$ and $u=0$ on $\{x\in \bar{\Omega} \:\vert    \: \dist(x , \Gamma) = 2 r_0 \}$, then
\begin{equation} \label{13}
\left(\int_{\mathcal{U}}\vert    u\vert    ^{q} \dx\right)^{\frac{1}{q}} \leq C \left(\int_{ \mathcal{U}}\vert    \nabla u\vert    ^{p} \dx\right)^{\frac{1}{p}}
\end{equation}
for each $ q \in \left[p, \frac{np}{n-p} \right)$, where $C=C(n,p,q,  \mathcal{U})$ is a positive constant.
\end{lemma}
\begin{proof}
The proof is by contradiction, considering that $W^{1, p}(\mathcal{U})$ is compactly embedded in $L^{q}(\mathcal{U})$.
\end{proof}

We define,
$$
V_{\ell,r}= \left\{x\in \Omega \:\vert    \: \left\vert     \dist \left( x , \Gamma \right) - \ell \right\vert    < r \right\}.
$$

\begin{proposition} \label{19}
Assume that the conditions  \eqref{4} - \eqref{9}, \eqref{24}  and \eqref{51}  are satisfied. Suppose that $u \in W^{1, p_1(\cdot)} _{\loc} (\bar{\Omega}\backslash $ $\Gamma )$ $\cap L^{\infty} _{\loc} \left(\bar{\Omega}\backslash \Gamma \right)$ satisfies
\begin{equation} \label{22}
\begin{split}
 \int_{\Omega}\left\langle A(\cdot , u , \nabla u) , \nabla \varphi \right\rangle + a( \cdot , u ) \varphi + &g( \cdot , u) \varphi  \dx  
+ \int _{\partial \Omega} b(\cdot , u) \varphi + h(\cdot , u) \varphi  \ds \leq 0,
\end{split}
\end{equation}
for all $\varphi \in W_{\loc}^{1, p_1(\cdot)}\left(\bar{\Omega} \backslash \Gamma \right) \cap L_{\loc}^{\infty}\left(\bar{\Omega} \backslash \Gamma \right)$, $\varphi \geq 0$, with  $\spt \varphi \subset \bar{\Omega} \backslash\Gamma$. Then, if $0<r < \ell < r_0$ we have the estimate
\begin{equation} \label{42}
\left\Vert       \max \{u,0\} \right\Vert      _{L^\infty \left(  V_{\ell, r/2} \right)}\leq C r^{-\tau},
\end{equation}
where $C=C\left(n , \mu , p_1, p_2 , q_1 , q_2 ,  \mathcal{U}\right)>0$ and $\tau=\tau\left(n , \mu , p_1, p_2 , q_1 , q_2 ,  \mathcal{U}\right)> \tfrac{\max \left\{p^+ _1 , q^+ _1 \right\}}{\min\left\{p_2^-  , q_2^- \right\} -\max \left\{p^+ _1 , q^+ _1 \right\} +1 }$. 

\end{proposition} 
\begin{proof}

1.  Let $u=C_\ast w$, where $C_\ast>1$ is a  number that will be determined below.  We assume that $\vert\{x\in V_{\ell , r/2}\:\vert \: w(x) > 0\}\vert \neq 0$, otherwise, \eqref{42} is immediate. Set $\Omega^\prime =\{x \in V_{\ell,r}  \:\vert    \: w(x)>0\}$. Take
$$
m_{t}=\esup \left\{w(x)\:\vert    \: x \in V_{\ell,t r} \cap \Omega^\prime \right\}, \quad 1/2 \leq t \leq 1 .
$$
Let $1/2 \leq s<t \leq 1$.  Define the functions $z:\Omega \rightarrow \mathbb{R}$, $z_k :\Omega \rightarrow \mathbb{R}$ by
\begin{align*}
z(x)&= w(x) - m_{t} \xi \left(\left\vert      \dist ( x ,\Gamma) - \ell \right\vert    \right),\\
z_k (x)&=\left\{ \begin{aligned}
& \max \left\{w(x)-m_{t} \xi \left(\left\vert      \dist ( x ,\Gamma) - \ell \right\vert    \right)  -k, 0\right\}  & & \text { if } x\in  V_{\ell , tr},\\
& 0 & & \text { if } x\in \Omega \backslash V_{\ell , tr},
\end{aligned}
\right. 
\end{align*}
where $0 \leq k \leq \esup _{\Omega^\prime} z$, and $\xi : \mathbb{R} \rightarrow \mathbb{R}$ is a smooth function satisfying: $\xi \equiv 0$ on $(-\infty , s r]$, $\xi \equiv 1$ on $ \left[ \frac{s+t}{2} r , \infty \right)$,
$$
0 \leq \xi \leq 1 \quad \text { and } \quad \left\vert     \xi ^{\prime}\right\vert     \leq \frac{C_1}{r(t-s)} \quad \text { on  }    \mathbb{R},
$$
where $C_1$ is a suitable positive constant. Observe that $z_{k} \in W^{1, p_1 (\cdot)} _{\loc}\left(\bar{\Omega}  \backslash \Gamma  \right)\cap L_{\loc}^{\infty}\left(\bar{\Omega} \backslash \Gamma \right)$ and  $ \spt z_k \subset \bar{\Omega} \cap \{ \left\vert      \dist ( \cdot ,\Gamma) - \ell \right\vert     \leq tr \}\subset \bar{\Omega} \backslash \Gamma$.  It is assumed that $m_{1/2} > 1$. The conclusion is obviously right for the case of $0<m_{1/2}\leq 1$. For simplicity we write $\xi\left( \left\vert      \dist ( x ,\Gamma) - \ell \right\vert     \right) = \xi (x)$ and $\xi ^\prime \left( \left\vert      \dist ( x ,\Gamma) - \ell \right\vert     \right) = \xi ^\prime (x)$. 

\noindent Take $k\in [0, \mathcal{K})$, where $ \mathcal{K}=\sup \{k\in [0, \esup _{\Omega ^\prime} z ]\:\vert \: \vert \{x \in V_{\ell, t r}\:\vert    \: z_{k} (x)>0\}\vert \neq 0\}$. Observe that $ \mathcal{K}\geq m_s \geq m_{1/2}>1$. Substituting $\varphi=z_{k}$ into \eqref{22}, we obtain
\begin{equation*} 
	\int_{\Omega}\left\langle  A(\cdot , u , \nabla u) , \nabla z_k \right\rangle +  a( \cdot , u ) z_k + g(\cdot , u) z_k  \dx  
	+\int _{ \partial \Omega } b(\cdot , u)z_k  + h(\cdot , u) z_k   \ds \leq 0.
\end{equation*}
Denote $\Omega_{k}=\{x \in V_{\ell, t r}\:\vert    \: z_{k} (x) >0\}$. Then,
\begin{align*}
& \int_{\Omega_{k}} \left\langle A\left(\cdot, C_\ast w, C_\ast \nabla w\right)  ,    \nabla w -m_{t} \xi ^{\prime}  \frac{  \dist ( \cdot ,\Gamma) - \ell }{ \left\vert      \dist ( \cdot ,\Gamma) - \ell \right\vert     }\nabla \dist (\cdot, \Gamma)\right\rangle \dx \\
&	+\int_{\Omega_{k}} g(\cdot , u)z_k \dx +  \int _{\partial \Omega _k  \cap \partial \Omega} h(\cdot , u)z_k \ds 
	\leq \int_{\Omega _k}\vert    a( \cdot , u )\vert     z_k  \dx  +  \int _{\partial \Omega _k  \cap \partial \Omega} \vert    b(\cdot , u)\vert     z_k  \ds .
\end{align*}
By \eqref{4} - \eqref{9}, we have
\begin{equation*} \label{36}
\begin{aligned}
&   \int_{\Omega_{k}}  \mu   C_{\ast} ^{p_1 - 1}\vert    \nabla w\vert    ^{p_1} \dx -  \int_{\Omega_{k}} \frac{\mu^{-1} C_1 m_t}{r(t-s)} \left(  \vert    C_\ast \nabla w\vert    ^{p_1-1} +   \vert    C_\ast w\vert    ^{p_1-1} + 1 \right) \dx \\
&   +\int_{\Omega_{k}} \left( \mu \vert    C_\ast w\vert    ^{p_2}  -  \mu ^{-1}  \right)z_k \dx + \int _{\partial \Omega _k  \cap \partial \Omega } \left( \mu \vert    C_\ast w\vert    ^{q_2} - \mu ^{-1}   \right)z_k \ds \\
&    \leq  \int_{\Omega _k} \mu ^{-1} \left(   \vert    C_\ast w\vert    ^{p_1 -1 }  + 1  \right)z_k \dx   +  \int _{\partial \Omega _k  \cap \partial \Omega } \mu ^{-1} \left( \vert    C_\ast w\vert    ^{q_1 -1}   +  1 \right)z_k \ds. 
\end{aligned}
\end{equation*}
Since $m_t > 1$, $C_\ast >1$, and observing that $w\geq k$ and  $m_t\geq w\geq z_k$ on $\Omega _k$,  we have
\begin{equation} \label{37}
\begin{aligned}
&   \int_{\Omega_{k}}     C_{\ast} ^{p_1 - 1}\vert    \nabla w\vert    ^{p_1} \dx + \int_{\Omega_{k}}   C_\ast ^{p_2} k^{p_2}z_k \dx + \int _{\partial \Omega _k  \cap \partial \Omega }   C_\ast ^{q_2} k^{q_2}z_k \ds \\
&   \leq  C_2\left(\mu \right)  \int_{\Omega_{k}} \frac{  m_t C_\ast ^{p_1-1} }{r(t-s)}   \vert    \nabla w\vert    ^{p_1-1} \dx \\
&   + C_2   C_\ast ^{\max \left\{p^+ _1 , q^+ _1 \right\} -1 }  \left[\frac{m_t}{r (t-s)} \right]^{\max \left\{p^+ _1 , q^+ _1 \right\}} \left(\vert    \Omega _k\vert     + \vert    \partial \Omega _k \cap \partial \Omega\vert    \right).
\end{aligned}
\end{equation}
On other hand, using Young's inequality,
\begin{equation} \label{38}
\begin{aligned}
   \int_{\Omega_{k}}   \frac{  m_t C_\ast ^{p_1-1} }{r(t-s)} \vert    \nabla w\vert    ^{p_1-1} \dx 
  \leq \int_{\Omega_{k}}   C_\ast ^{p_1-1} \left\{ C_3 \left(\varepsilon _1 , p_1 \right)\left[\frac{m_t}{r(t-s)}\right]^{p_1} + \varepsilon _1 \vert    \nabla w\vert    ^{p_1}\right\} \dx.
\end{aligned}
\end{equation}
Take  $\varepsilon _1 = \frac{1}{2C_2 }$. Then, by \eqref{37} and \eqref{38},
\begin{align*}
&    \frac{1 }{2}\int_{\Omega_{k}}     C_{\ast} ^{p_1 - 1}  \vert    \nabla w\vert    ^{p_1} \dx + \int_{\Omega_{k}}   C_\ast ^{p_2} k^{p_2}z_k \dx + \int _{\partial \Omega _k  \cap \partial \Omega }   C_\ast ^{q_2} k^{q_2}z_k \ds \\
&     \leq    C_4\left(\mu  ,  p_1 \right)C_\ast ^{\max \left\{p^+ _1 , q^+ _1 \right\} -1 } \left[\frac{m_t}{r (t-s)} \right]^{\max \left\{p^+ _1 , q^+ _1 \right\}} \left(\vert    \Omega _k\vert     + \vert    \partial \Omega _k \cap \partial \Omega\vert    \right).
\end{align*}
Note that $\nabla z_{k}=\nabla w-m_{t}  \xi ^{\prime}  \frac{  \dist ( \cdot ,\Gamma) - \ell }{ \left\vert      \dist ( \cdot ,\Gamma) - \ell \right\vert     }\nabla \dist (\cdot, \Gamma)$ in $\Omega_{k}$, therefore
\begin{equation}\label{11}
\begin{aligned}
&   \frac{1 }{2}\int_{\Omega_{k}}  C_{\ast} ^{p_1 - 1}  \left(\frac{1}{2^{p_1-1}}\left\vert    \nabla z_{k}\right\vert    ^{p_1}-m_{t}^{p_1}\left\vert    \xi ^\prime\right\vert    ^{p_1}\right)  \dx 
	+ \int_{\Omega_{k}}   C_\ast ^{p_2} k^{p_2}z_k \dx + \int _{\partial \Omega _k  \cap \partial \Omega }  C_\ast ^{q_2} k^{q_2}z_k \ds \\
&    \leq   C_4 C_\ast ^{\max \left\{p^+ _1 , q^+ _1 \right\} -1 } \left[\frac{m_t}{r (t-s)} \right]^{\max \left\{p^+ _1 , q^+ _1 \right\}} \left(\vert    \Omega _k\vert     + \vert    \partial \Omega _k \cap \partial \Omega\vert    \right).
\end{aligned}
\end{equation}
We have
\begin{equation}\label{12}
  \left\vert    \nabla z_{k}\right\vert    ^{p_1^{-}}  \leq  1 + \left\vert    \nabla z_{k}\right\vert    ^{p_1} .
\end{equation}
Then, from \eqref{11} and \eqref{12},
\begin{equation} \label{14}
\begin{aligned}
&    \frac{1}{2^{p^+ _1}} \int _{\Omega_{k}}    C_{\ast} ^{p_1 - 1}  \left(\left\vert    \nabla z_{k}\right\vert    ^{p_1 ^- }  - 1 \right) \dx + \int_{\Omega_{k}}   C_\ast ^{p_2} k^{p_2}z_k \dx + \int _{\partial \Omega _k  \cap \partial \Omega }   C_\ast ^{q_2} k^{q_2}z_k \ds \\
&    \leq    C_5\left(\mu , p_1  \right)C_\ast ^{\max \left\{p^+ _1 , q^+ _1 \right\} -1 } \left[\frac{m_t}{r (t-s)} \right]^{\max \left\{p^+ _1 , q^+ _1 \right\}} \left(\vert    \Omega _k\vert     + \vert    \partial \Omega _k \cap \partial \Omega\vert    \right).
\end{aligned}
\end{equation}

On the other hand, let  $\gamma  \in \left(p_1 ^- , \frac{(n-1)p_1 ^-}{n-p_1 ^-} \right]$ and $\alpha \in \left(p_1 ^-,\gamma \right)$ be arbitrary.  Observe that  $\spt z_k \cap \Omega \subset  \Omega \cap \{ \left\vert     \dist (\cdot , \Gamma ) - \ell \right\vert     \leq tr\} \subset \mathcal{U}$. By  \eqref{13}, and since the trace $W^{1, p_1 ^-} (\mathcal{U}) \rightarrow L^\gamma (\partial \mathcal{U})$ is continuous,
\begin{equation}\label{39} 
	\left(\int_{V_{\ell , t r}}z_{k}^{\alpha} \dx\right)^{\frac{1}{\alpha}}   + \left(\int_{\partial V_{\ell , t r}\cap \partial \Omega} z_{k}^{\alpha} \ds \right)^{\frac{1}{\alpha} }
	\leq  C_6 r^{\frac{(n-d-1)(\gamma  - \alpha)}{\gamma \alpha}} \left(\int _{V_{\ell , t r}}\left\vert    \nabla z_{k}\right\vert    ^{p_1^{-}} \dx\right)^{\frac{1}{p_1 ^-}},
\end{equation}
where $C_6 = C_6\left(n, \mu ,    p_1  ,  \gamma , \mathcal{U}\right)>0$.  Using Hölder's inequality, we get
\begin{align}
\int_{V_{\ell, t r}} z_{k} \dx & \leq \left(\int_{V_{\ell, t r}} z_{k}^{ \alpha} \dx\right)^{\frac{1}{\alpha}}\left\vert    \Omega_{k}\right\vert    ^{\frac{\alpha -1}{ \alpha }} \label{40}, \\
\int_{\partial V_{\ell, t r} \cap \partial \Omega} z_{k} \ds & \leq \left(\int_{\partial V_{\ell, t r} \cap \partial \Omega} z_{k}^{ \alpha} \ds \right)^{\frac{1}{\alpha}}\left\vert    \partial \Omega_{k} \cap \partial \Omega \right\vert    ^{\frac{\alpha - 1}{ \alpha}}\label{41}.
\end{align}
From \eqref{14} - \eqref{41}, we have
\begin{equation} \label{15}
\begin{aligned}
&	\frac{C_\ast  ^{p_1^-  - 1}}{2^{p^+ _1}}   \left[  C_6 ^{-1} \left(\left\vert    \Omega _k \right\vert     + \left\vert    \partial \Omega _k \cap \partial \Omega\right\vert    \right)^{-\frac{\alpha -  1}{\alpha}}  r^{-\frac{(n-d-1)(\gamma - \alpha)}{\gamma \alpha}} 
 \left(  \int_{V_{\ell , tr}} z_k \dx +  \int_{\partial V_{\ell , tr} \cap \partial \Omega} z_k \ds\right) \right] ^{p_1^-}  \\
&	+ C_\ast ^{\min\left\{p_2^- , q_2^- \right\}}\min \left\{ k^{p_2^-} , k^{p_2^+ , } , k^{q_2^-} , k^{q_2^+}\right \}  \left(\int _{V_{\ell , tr}} z_k \dx + \int _{\partial V_{\ell , tr} \cap \partial \Omega} z_k \ds \right)\\
&	\leq    C_5 C_\ast ^{\max \left\{p^+ _1 , q^+ _1 \right\} -1 } \left[\frac{m_t}{r (t-s)} \right]^{\max \left\{p^+ _1 , q^+ _1 \right\}} \left(\vert    \Omega _k\vert     + \vert    \partial \Omega _k \cap \partial \Omega\vert    \right)+\frac{C_{\ast} ^{p_1 ^+ -1}}{2^{p_1 ^+}}  \left\vert    \Omega _k\right\vert    .
\end{aligned}
\end{equation}\\

2.  Take $\varepsilon \in (0,1)$. Then, \eqref{15} implies
\begin{equation} \label{16}
\begin{aligned}
&	\varepsilon   C_\ast ^{p_1^-  - 1}\left(   \int_{V_{\ell , tr}} z_k \dx  +  \int_{\partial V_{\ell , tr} \cap \partial \Omega} z_k \ds \right)^{p_1^-}  r^{-p_1^-\frac{(n-d-1)(\gamma - \alpha)}{\gamma \alpha}} \\
&	+ (1-\varepsilon) C_\ast ^{\min\left\{p_2^- , q_2^- \right\}}\min \left\{ k^{p_2^-} , k^{p_2^+ , } , k^{q_2^-} , k^{q_2^+}\right \}\\
&	\cdot \left(\left\vert    \Omega _k \right\vert     + \left\vert    \partial \Omega _k \cap \partial \Omega\right\vert    \right)^{p_1^-\frac{\alpha -  1}{\alpha}} \left(\int _{V_{\ell , tr}} z_k \dx + \int _{\partial V_{\ell , tr} \cap \partial \Omega} z_k \ds \right)\\
&	\leq  C_7 C_\ast ^{\max \left\{p^+ _1 , q^+ _1 \right\} -1 } \left\{\left[\frac{m_t}{r (t-s)} \right]^{\max \left\{p^+ _1 , q^+ _1 \right\}} +1\right\}  \left(\left\vert    \Omega _k \right\vert     + \left\vert    \partial \Omega _k \cap \partial \Omega\right\vert    \right)^{1+p_1^- \frac{\alpha -  1}{\alpha}},
\end{aligned}
\end{equation}
where $C_7 = C_7 \left(C_5 , C_6 ,\mu , p_1\right)>0$. Applying Young's inequality $a^{\varepsilon} b^{1-\varepsilon} \leq \varepsilon a+(1-\varepsilon) b$ in the left-hand side of \eqref{16}, we obtain
\begin{equation*} 
\begin{aligned}
&	\min  \left\{ k^{p_2^- (1-\varepsilon)} \right.  ,  \left. k^{p_2^+ (1-\varepsilon) } , k^{q_2^- (1-\varepsilon)} , k^{q_2^+ (1-\varepsilon)}\right \}   C_\ast ^{(p_1^-  - 1)\varepsilon + \min\left\{p_2^- , q_2^- \right\}(1-\varepsilon)} \\
&	\cdot r^{-p_1^-\frac{(n-d-1)(\gamma - \alpha)}{\gamma \alpha}\varepsilon} \left(   \int_{V_{\ell , tr}} z_k \dx  +  \int_{\partial V_{\ell , tr} \cap \partial \Omega} z_k \ds \right)^{p_1^- \varepsilon + 1 - \varepsilon} \\
&	\leq    C_7 C_\ast ^{\max \left\{p^+ _1 , q^+ _1 \right\} -1 } \left\{\left[\frac{m_t}{r (t-s)} \right]^{\max \left\{p^+ _1 , q^+ _1 \right\}} +1\right\}  \left(\left\vert    \Omega _k \right\vert     + \left\vert    \partial \Omega _k \cap \partial \Omega\right\vert    \right)^{1+p_1^- \frac{\alpha -  1}{\alpha}\varepsilon}.
\end{aligned}
\end{equation*}
Therefore, 
\begin{equation} \label{17}
\begin{aligned}
&	\min  \left\{  k^{\frac{p_2^- (1-\varepsilon)}{\beta}} \right. , \left.k^{\frac{p_2^+ (1-\varepsilon)}{\beta} } , k^{\frac{q_2^- (1-\varepsilon)}{\beta}} , k^{\frac{q_2^+ (1-\varepsilon)}{\beta}}\right\}  
	C_\ast ^{\frac{\min\left\{p_2^- , q_2^- \right\} -\max \left\{p^+ _1 , q^+ _1 \right\} +1 - \left( \min\left\{p_2^- , q_2^- \right\} - p_1^-  +1\right)\varepsilon}{\beta} } \\
&	\leq   (2C_7) ^{\frac{1}{\beta}} \left(   \int_{\Omega _k} z_k \dx  +  \int_{\partial \Omega _k \cap \partial \Omega} z_k \ds \right)^{-\frac{p_1^- \varepsilon + 1 - \varepsilon }{\beta}} 
 r^{p_1^-\frac{(n-d-1)(\gamma - \alpha)}{\gamma \alpha}\frac{\varepsilon}{\beta}} \left[\frac{m_t}{r (t-s)} \right]^{\frac{\max \left\{p^+ _1 , q^+ _1 \right\}}{\beta}}   \left(\left\vert    \Omega _k \right\vert     + \left\vert    \partial \Omega _k \cap \partial \Omega\right\vert    \right),
\end{aligned}
\end{equation} 
where  $\beta =1+p_1^- \frac{\alpha -  1}{\alpha}\varepsilon$.\\

3.  Integrating \eqref{17} with respect to $k$,

\begin{equation*} 
\begin{aligned}
&	C_\ast  ^{\frac{\min\left\{p_2^-  , q_2^- \right\} -\max \left\{p^+ _1 , q^+ _1 \right\} +1 - \left( \min\left\{p_2^- , q_2^- \right\} - p_1^-  +1\right)\varepsilon }{\beta}}  
 \int _0 ^{\mathcal{K}} \min  \left\{  k^{\frac{p_2^- (1-\varepsilon)}{\beta}} \right. , \left.k^{\frac{p_2^+ (1-\varepsilon)}{\beta} } , k^{\frac{q_2^- (1-\varepsilon)}{\beta}} , k^{\frac{q_2^+ (1-\varepsilon)}{\beta}}\right\}  \textnormal{d}k \\
&	\leq  (2 C_{7} )^{\frac{1}{\beta}}  r^{p_1^-\frac{(n-d-1)(\gamma - \alpha)}{\gamma \alpha}\frac{\varepsilon}{\beta}} \left[\frac{m_t}{r (t-s)} \right]^{\frac{\max \left\{p^+ _1 , q^+ _1 \right\}}{\beta}}  \\
&	\cdot \int _0 ^{\mathcal{K}} \left(   \int_{\Omega _k} z_k \dx  +  \int_{\partial \Omega _k \cap \partial \Omega} z_k \ds \right)^{-\frac{p_1^- \varepsilon + 1 - \varepsilon }{\beta}}\left(\left\vert    \Omega _k \right\vert     + \left\vert    \partial \Omega _k \cap \partial \Omega\right\vert    \right) \textnormal{d}k.
\end{aligned}
\end{equation*} 
Let us consider the equalities  
$$
\frac{\textnormal{d}}{\textnormal{d} k}\left(\int_{\Omega_{k}} z_{k} \dx\right)=-\left\vert    \Omega_{k}\right\vert    \quad \text { and } \quad \frac{\textnormal{d}}{\textnormal{d} k}\left(\int_{\partial \Omega_{k} \cap \partial \Omega} z_{k} \ds\right)=-\left\vert    \partial \Omega_{k} \cap \partial \Omega\right\vert    .
$$
Since $1-\frac{p_1^- \varepsilon + 1 - \varepsilon }{\beta}>0$ and  $\mathcal{K}>1$, 
\begin{equation*}
\begin{aligned}
&	\mathcal{K} ^{1+ \frac{\min \left\{ p_2 ^- , q_2 ^-\right\}(1-\varepsilon)}{\beta} } C_\ast  ^{\frac{\min\left\{p_2^-  , q_2^- \right\} -\max \left\{p^+ _1 , q^+ _1 \right\} +1 - \left( \min\left\{p_2^- , q_2^- \right\} - p_1^-  +1\right)\varepsilon }{\beta}} \\
&	\leq  \varepsilon ^{-1} C_{8}   r^{p_1^-\frac{(n-d-1)(\gamma - \alpha)}{\gamma \alpha}\frac{\varepsilon}{\beta}}       \left[\frac{m_t}{r (t-s)} \right]^{\frac{\max \left\{p^+ _1 , q^+ _1 \right\}}{\beta}}  
  \left(   \int_{\Omega _0} z_0 \dx  +  \int_{\partial \Omega _0 \cap \partial \Omega} z_0 \ds \right)^{1-\frac{p_1^- \varepsilon + 1 - \varepsilon }{\beta}},
\end{aligned}
\end{equation*}
where $C_{8}=C_{8} \left(n, \mu ,  \alpha , \gamma , p_1 , p_2 , q_2 , \mathcal{U}\right)>0$. We note  $\mathcal{K} \geq m_{s}$. Apply the estimates 
$$
\int_{\Omega_{0}} z_{0} \dx \leq m_{t}\left\vert    V_{\ell,r}\right\vert    \quad \text { and } \quad \int_{\partial \Omega _0 \cap \partial \Omega} z_0 \ds \leq m_{t}\left\vert    \partial V_{\ell , r} \cap \partial \Omega\right\vert     ,
$$ 
we have
\begin{equation*}
\begin{aligned}
&	m_s ^{1+ \frac{\min \left\{ p_2 ^- , q_2 ^-\right\}(1-\varepsilon)}{\beta} } C_\ast  ^{\frac{\min\left\{p_2^-  , q_2^- \right\} -\max \left\{p^+ _1 , q^+ _1 \right\} +1 - \left( \min\left\{p_2^- , q_2^- \right\} - p_1^-  +1\right)\varepsilon }{\beta}} \\
&	\leq   C_{9}   r^{p_1^-\frac{(n-d-1)(\gamma - \alpha)}{\gamma \alpha}\frac{\varepsilon}{\beta}}       \left[\frac{m_t}{r (t-s)} \right]^{\frac{\max \left\{p^+ _1 , q^+ _1 \right\}}{\beta}}  
  m_t ^{1-\frac{p_1^- \varepsilon + 1 - \varepsilon }{\beta}} r^{(n-d-1)\left( 1-\frac{p_1^- \varepsilon + 1 - \varepsilon }{\beta}\right)},
\end{aligned}
\end{equation*}
where  $C_{9}=C_{9} \left(n , \mu  , \varepsilon  , \alpha , \gamma , p_1 , p _2 , q _2 , \mathcal{U} \right)>0$. 

Now we take  $C_\ast>0$ such that
$$
	C_\ast  ^{\frac{\min\left\{p_2^-  , q_2^- \right\} -\max \left\{p^+ _1 , q^+ _1 \right\} +1 - \left( \min\left\{p_2^- , q_2^- \right\} - p_1^-  +1\right)\varepsilon }{\beta}} 
	= r^{p_1 ^- \frac{(n-d-1)(\gamma - \alpha)}{\gamma \alpha}\frac{\varepsilon}{\beta} -\frac{\max \left\{p_1 ^+ , q_1 ^+\right\}}{\beta} + (n-d-1)\left( 1-\frac{p_1^- \varepsilon + 1 - \varepsilon }{\beta}\right)}.
$$
Hence $C_\ast = r^{-\tau}$, where 
$$
\begin{aligned}
&	\tau  = -\left[ p_1^-\frac{(n-d-1)(\gamma - \alpha)}{\gamma \alpha}\frac{\varepsilon}{\beta} - \frac{\max \left\{p^+ _1 , q^+ _1 \right\}}{\beta} + (n-d-1)\left( 1-\frac{p_1^- \varepsilon + 1 - \varepsilon }{\beta}\right)\right]\\
&	\cdot\left[\frac{\min\left\{p_2^-  , q_2^- \right\} -\max \left\{p^+ _1 , q^+ _1 \right\} +1 - \left( \min\left\{p_2^- , q_2^- \right\} - p_1^-  +1\right)\varepsilon }{\beta}\right]^{-1}\\
&	> \frac{\max \left\{p^+ _1 , q^+ _1 \right\}}{\min\left\{p_2^-  , q_2^- \right\} -\max \left\{p^+ _1 , q^+ _1 \right\} +1 }>0,
\end{aligned}
$$
if $1> (n-d-1)\left(1 - \frac{p_1^-}{\gamma}\right) $ and $\varepsilon \in \left(0 , \frac{\min\left\{p_2^-  , q_2^- \right\} -\max \left\{p^+ _1 , q^+ _1 \right\} +1}{\min\left\{p_2^- , q_2^- \right\} - p_1^-  +1} \right)$.

On the other hand,
$$
m_{s} \leq C_{9}^{\frac{\beta}{\max \left\{p^+ _1 , q^+ _1 \right\}}\sigma} \frac{m_{t}^{\theta}}{(t-s)^{\sigma}},
$$
where
$$
\begin{aligned}
\theta=&\left[\frac{\max \left\{p^+ _1 , q^+ _1 \right\}}{\beta}+ 1-\frac{p_1^- \varepsilon + 1 - \varepsilon }{\beta}\right] \left[1+ \frac{\min \left\{ p_2 ^- , q_2 ^-\right\}(1-\varepsilon)}{\beta}\right]^{-1} 
<1 , \\
\sigma=&\left[\frac{\max \left\{p^+ _1 , q^+ _1 \right\}}{\beta}\right] \left[1+ \frac{\min \left\{ p_2 ^- , q_2 ^-\right\}(1-\varepsilon)}{\beta}\right]^{-1}. 
\end{aligned}
$$
By virtue of Lemma \ref{18}, we derive 
$$
m_{1/2} \leq C_{10} \left(n , \mu ,   p _1 , p_2  , q _1 ,  q _2  , \mathcal{U} \right).
$$
From the substitution $u=C_\ast w$ we obtain
$$
\esup \left\{u(x)\:\vert    \: x \in V_{\ell , r/2} \cap \Omega^{\prime}\right\}=C_\ast m_{1/2} \leq C_{10} C_\ast =  C_{10} r^{-\tau}.
$$
Therefore, we  conclude  the proof of the Proposition \ref{19}.
\end{proof}

The next theorem follows easily from the Proposition \ref{19}.

\begin{theorem}\label{21}
Suppose that the conditions \eqref{4} - \eqref{24} and \eqref{51} are satisfied. Let 
$u \in W^{1, p_1 (\cdot)} _{\loc} \left(\bar{\Omega}\backslash \Gamma \right)$ $\cap L^{\infty} _{\loc} \left(\bar{\Omega}\backslash \Gamma \right)$ be a solution of equation \eqref{1} in $\bar{\Omega} \backslash \Gamma$. Then, in $\left\{x\in \Omega \:\vert    \: 0<\dist ( x , \Gamma ) < r_0\right\}$, the following inequality holds almost everywhere:
\begin{equation}\label{25}
\vert    u(x)\vert     \leq C \dist (x,\Gamma)^{-\tau} ,
\end{equation}
where $C=C\left(n , \mu , p_1 ,  p_2 , q_1 ,  q_2 , \mathcal{U}\right)>0$ and $\tau=\tau \left(n , \mu , p_1 ,  p_2 , q_1 ,  q_2 , \mathcal{U}\right)> \frac{\max \left\{p^+ _1 , q^+ _1 \right\}}{\min\left\{p_2^-  , q_2^- \right\} -\max \left\{p^+ _1 , q^+ _1 \right\} +1 }$. 
\end{theorem}

Proceeding in the same way as in  the Proposition \ref{19} and Theorem \ref{21}, we have the following results.

\begin{proposition} \label{26} 
Suppose that the conditions \eqref{4} - \eqref{7}  and \eqref{51} are satisfied, additionally
\begin{equation} \label{50}
p_2 ^- - p_1 ^+ +1 >0.
\end{equation} 
Assume that $u \in W^{1, p_1 (\cdot)} _{\loc} \left(\bar{\Omega}\backslash \Gamma \right)$ $\cap L^{\infty} _{\loc} \left(\bar{\Omega}\backslash \Gamma \right)$ satisfies
\begin{equation} \label{35}
 \int_{\Omega}\left\langle A(\cdot , u , \nabla u) , \nabla \varphi \right\rangle + a( \cdot , u ) \varphi + g(\cdot , u) \varphi  \dx \leq 0,
\end{equation}
for all $\varphi \in W_{\loc}^{1, p_1 (\cdot)}\left(\bar{\Omega} \backslash \Gamma \right) \cap L_{\loc}^{\infty}\left(\bar{\Omega} \backslash \Gamma \right)$, $\varphi \geq 0$, with $\spt \varphi \subset \bar{\Omega} \backslash\Gamma$. Then, if $0<r < \ell < r_0$ we have the estimate
$$
\left\Vert      \max \{u,0\} \right\Vert      _{L^\infty \left( V_{\ell, r/2}\right)} \leq C r^{-\tau},
$$
where $C=C\left(n , \mu , p_1 ,  p_2 , \mathcal{U}\right)>0$ and $\tau=\tau \left(n , \mu , p_1 ,  p_2 , \mathcal{U}\right)$  $> \frac{p^+ _1}{p_2^-  - p^+ _1 +1 }$. 
\end{proposition}

\begin{theorem}\label{27} 
Suppose that the conditions \eqref{4} - \eqref{7}, \eqref{23}, \eqref{51} and \eqref{50}  are satisfied. Let $u \in W^{1, p_1 (\cdot)} _{\loc} \left(\bar{\Omega}\backslash \Gamma \right)$ $\cap L^{\infty} _{\loc} \left(\bar{\Omega}\backslash \Gamma \right)$ be a solution of equation \eqref{1} in $\bar{\Omega} \backslash \Gamma$, with $b\equiv h \equiv 0$. Then, in $\left\{x\in \Omega \:\vert    \: 0<\dist ( x , \Gamma ) < r_0\right\}$, the following inequality  holds almost everywhere:
\begin{equation*}
\vert    u(x)\vert     \leq C \dist (x,\Gamma)^{-\tau},
\end{equation*}
where $C=C\left(n , \mu , p_1 ,  p_2 , \mathcal{U}\right)>0$ and $\tau=\tau \left(n , \mu , p_1 ,  p_2 , \mathcal{U}\right)$ $> \frac{p^+ _1}{p_2^-  - p^+ _1 +1 }$. 
\end{theorem}

\section{The removability of singular set}

Next, we prove the main theorem of this paper. Before, we start with the following 

\begin{lemma} \label{29}
Suppose that the conditions  \eqref{4} - \eqref{9}, \eqref{24} - \eqref{51} are satisfied.  If $u \in W^{1, p_1(\cdot)} _{\loc} \left(\bar{\Omega}\backslash \Gamma \right)$ $\cap L^{\infty} _{\loc} \left(\bar{\Omega}\backslash \Gamma \right)$ satisfies \eqref{22}, then
$$
\max \{u , 0\} \in L^{\infty} (\Omega).
$$

\end{lemma}
\begin{proof}
We proceed by contradiction. For $r\in \left(0,r_0 ^2\right)$, we denote
$$
\Lambda(r)=\esup \left\{\max\{u(x),0\}\:\vert    \: r \leq \dist (x,\Gamma) \leq r_0^2 , \ x\in \Omega\right\}.
$$
We have  $\lim _{r \rightarrow 0^+} \Lambda (r)=\infty$. For sufficiently small values $r$ we define the function $\psi_{r}:\mathbb{R}\rightarrow \mathbb{R}$ as follows:
$$
\psi_{r}(t) = \left\{
\begin{aligned}
 & 0 & & \text { if } t<r, \\
 & 1 & & \text { if  } t >\sqrt{r}, \\
& \frac{2}{\ln \frac{1}{r}} \ln \frac{t}{r} & & \text { if } r\leq t \leq \sqrt{r} .
\end{aligned}
\right.
$$
Choosing $\delta >0$ such that $\Lambda (\delta)>1$, set
$$
\varphi = \left( \ln \  \max\left\{ \frac{u}{\Lambda  (\delta)}  , 1   \right\} \right) \psi _r ^{\gamma} \circ \dist (\cdot,\Gamma),
$$
where $\gamma=\esup _{\Omega}\frac{p_1 p_2}{p_2 - p_1 +1}$. We have  $\varphi \in W^{1, p_1(\cdot)} _{\loc} \left(\bar{\Omega}\backslash \Gamma \right)\cap L^\infty _{\loc} \left(\bar{\Omega}\backslash \Gamma \right)$, with  $\spt \varphi \subset \bar{\Omega}\cap \{ \dist (\cdot , \Gamma) \geq r\} \subset \bar{\Omega} \backslash \Gamma$. For simplicity we write $\psi _r  \circ \dist (\cdot,\Gamma) = \psi _r $ and $\psi _r ^\prime \circ \dist (\cdot,\Gamma) = \psi _r ^\prime$. 

Denote $\Omega _{\delta} = \left\{x\in \Omega \:\vert    \: u(x)> \Lambda (\delta) \right\}$.  Substituting $\varphi$ into \eqref{22}, we obtain

\begin{equation*}
\begin{split}
&	\int_{\Omega _{\delta} }  \left[\frac{\psi _r ^{\gamma} }{u}\left\langle A(\cdot, u, \nabla u) , \nabla u  \right\rangle \right. + \gamma \psi _r ^\prime \psi _r ^{\gamma -1}   \left( \ln \frac{u}{\Lambda  (\delta)} \right)  \left\langle A(\cdot , u , \nabla u) , \nabla \dist (\cdot , \Gamma )  \right\rangle   \\
&	\left.+ a( \cdot , u ) \psi _r ^{\gamma}   \left( \ln \frac{u}{\Lambda  (\delta)} \right)  + g(\cdot , u)  \psi _r ^{\gamma}   \left( \ln \frac{u}{\Lambda  (\delta)} \right)  \right] \dx  \\
&	+ \int _{\partial \Omega _{\delta} \cap \partial \Omega} \left[ b(\cdot,u) \psi _r ^{\gamma}   \left( \ln \frac{u}{\Lambda  (\delta)} \right)  + h(\cdot,u)  \psi _r ^{\gamma}   \left( \ln \frac{u}{\Lambda  (\delta)} \right)  \right] \ds \leq 0.
\end{split}
\end{equation*}
Since $u> \Lambda (\delta)>1$ in $\Omega _{\delta}$, and by virtue of the conditions \eqref{4} - \eqref{9}, we have
\begin{equation*}
\begin{split}
&	\int_{\Omega _{\delta} }  \mu \frac{\psi _r ^{\gamma} }{u} \vert    \nabla u\vert     ^{p_1} \dx +  \int _{\Omega _{\delta}} \left(\mu -  \frac{3\mu ^{-1}}{\Lambda (\delta) ^{p^- _2 - p^+ _1 +1}} \right) \psi _r ^{\gamma}   \left( \ln \frac{u}{\Lambda  (\delta)} \right)  u^{p_2} \dx\\
&	+  \int _{\partial \Omega _{\delta} \cap \Omega } \left(\mu  - \frac{3\mu ^{-1}}{\Lambda (\delta) ^{q^- _2 - q^+ _1 +1}}\right) \psi _r ^{\gamma}  \left( \ln \frac{u}{\Lambda  (\delta)} \right)  u^{q _2} \ds \\
&	\leq \int_{\Omega _{\delta} } \mu ^{-1}\gamma \psi _r ^\prime \psi _r ^{\gamma -1}   \left( \ln \frac{u}{\Lambda  (\delta)} \right) \left(  \vert    \nabla u\vert    ^{p_1 - 1} + u ^{p_1 - 1} + 1\right)\dx\\
&	\leq   \int_{\Omega _{\delta} } \mu ^{-1} \gamma \psi _r ^{\gamma}   u^{-1} \left\{ C_1 \left(\varepsilon _1 , p_1 \right) \left[ \psi _r ^\prime \psi _r ^{ -1}  u \left( \ln \frac{u}{\Lambda  (\delta)} \right)  \right] ^{p_1} + \varepsilon _1  \vert    \nabla u\vert    ^{p_1} \right\}\dx\\
&	+ \int_{\Omega _{\delta} } \mu ^{-1}\gamma  \left( \ln \frac{u}{\Lambda  (\delta)} \right) \left[ \left(C_2 \left(\varepsilon _2 , p_1 , p_2 \right) + 1\right)(\psi _r ^{\prime})^{\frac{p_2}{p_2 - p_1 + 1}}  + \varepsilon _2\psi _r ^{\frac{(\gamma -1)p_2}{p_1 -1}} u^{p_2} + \psi _r ^{\frac{(\gamma -1)p_2}{p_1 -1}} \right]\dx.
\end{split}
\end{equation*}
We can assume that 
$$
\varepsilon _3 = \mu  - \frac{\mu ^{-1} (3+\gamma)}{\Lambda (\delta) ^{\min \{p^- _2 , q_2 ^-\} - \max\{p^+ _1 , q^+ _1\} +1}}   >0,
$$
because $\lim _{r\rightarrow 0 ^+}\Lambda (r) = \infty$. Take $\varepsilon _1 = \frac{\mu^2}{2 \gamma}$, $\varepsilon _2 =\frac{\mu \varepsilon _3}{2\gamma}$. Since $\frac{(\gamma -1)p_2}{p_1 -1} \geq \gamma$ and $\psi _r \leq 1$,
\begin{equation}\label{28}
\begin{split}
&   \int_{\Omega _{\delta} }  \frac{\psi _r ^{\gamma} }{u} \vert    \nabla u\vert     ^{p_1} \dx +  \int _{\Omega _{\delta}} \psi _r ^{\gamma}   \left( \ln \frac{u}{\Lambda  (\delta)} \right)  u^{p_2} \dx +  \int _{\partial \Omega _{\delta} \cap \Omega }  \psi _r ^{\gamma}  \left( \ln \frac{u}{\Lambda  (\delta)} \right)  u^{q _2} \ds \\
&   \leq  C_3  \int_{\Omega _{\delta} }  \left(\psi _r ^\prime \right) ^{p_1}\psi _r ^{\gamma - p_1}  u ^{p_1 -1} \left( \ln \frac{u}{\Lambda  (\delta)} \right) ^{p_1} +  \left( \ln \frac{u}{\Lambda  (\delta)} \right) (\psi _r ^{\prime})^{\frac{p_2}{p_2 - p_1 + 1}} \dx,
\end{split}
\end{equation}
where $C_3 = C_3 \left(\mu , \gamma , \varepsilon _3 , p_1 , p_2\right)>0$.

Additionally, let us consider
$$
\begin{aligned}
&   \left(\psi _r ^\prime \right)  ^{p_1}\psi _r ^{\gamma - p_1}  u ^{p_1 -1} \left( \ln \frac{u}{\Lambda  (\delta)} \right) ^{p_1} \\
&   \leq  \left( \ln \frac{u}{\Lambda  (\delta)} \right) \left\{ C_4 \left(\varepsilon _4 , p_1 , p_2 \right) \left[ \left(\psi _r ^\prime \right) ^{p_1}  \left( \ln \frac{u}{\Lambda  (\delta)} \right) ^{p_1 -1} \right]^{\frac{p_2}{p_2 - p_1 +1}} + \varepsilon _4 \psi _r ^{\frac{(\gamma - p_1)p_2}{p_1 -1}}  u ^{p_2} \right\},
\end{aligned}
$$ 
$$
\frac{(\gamma - p_1)p_2}{p_1 -1} \geq \gamma \quad \text { and } \quad  \psi _r ^\prime (t) = \frac{2}{t\ln \frac{1}{r}} >1 \text { for } t\in \left(r,\sqrt{r}\right).
$$ 
Choose $\varepsilon _4 =\frac{1}{2 C_3}$. By \eqref{28},
\begin{equation*}
\begin{split}
&    \int_{\Omega _{\delta} }   \frac{\psi _r ^{\gamma} }{u} \vert    \nabla u\vert     ^{p_1} \dx +  \int _{\Omega _{\delta}}  \psi _r ^{\gamma}   \left( \ln \frac{u}{\Lambda  (\delta)} \right)  u^{p_2} \dx +  \int _{\partial \Omega _{\delta} \cap \Omega }  \psi _r ^{\gamma}  \left( \ln \frac{u}{\Lambda  (\delta)} \right)  u^{q _2} \ds \\
&    \leq   C_5\left(\mu , \gamma , \varepsilon _3 , p_1 , p_2 \right)  \int_{\Omega _{\delta} \cap \{r\leq \dist ( \cdot ,\Gamma)\leq \sqrt{r}\}} \left( \ln \frac{u}{\Lambda  (\delta)} \right) 
 \left[  \left( \ln \frac{u}{\Lambda  (\delta)} \right) ^{\frac{(p_1 -1) p_2}{p_2 - p_1 +1}} + 1 \right] 
     \left(\frac{2}{\dist (\cdot , \Gamma) \ln \frac{1}{r}}\right) ^{\frac{p_1 p_2}{p_2 - p_1 +1}}\dx.
\end{split}
\end{equation*}
We can assume that $1<\ln \frac{1}{\dist ( x , \Gamma)}$ if $0<\dist ( x , \Gamma ) \leq \sqrt{r}$. Using \eqref{25}, we see
\begin{equation*}
\begin{split}
& \int_{\Omega _{\delta} }    \frac{\psi _r ^{\gamma} }{u} \vert    \nabla u\vert     ^{p_1} \dx +  \int _{\Omega _{\delta}}  \psi _r ^{\gamma}   \left( \ln \frac{u}{\Lambda  (\delta)} \right)  u^{p_2} \dx  \\
& \leq  C_6 \left(\ln \frac{1}{r} \right) ^{-\frac{p_1 ^- p_2 ^-}{p_2 ^+ - p_1 ^- +1}}\int_{\Omega _{\delta} \cap \{r\leq \dist ( \cdot ,\Gamma)\leq \sqrt{r}\}} \left( \ln \frac{1}{ \dist (\cdot , \Gamma) } \right) ^{1+\frac{(p_1 -1) p_2}{p_2 - p_1 +1}}
 \left(\frac{1}{\dist (\cdot , \Gamma)}\right) ^{\frac{p_1 p_2}{p_2 - p_1 +1}}\dx\\
& \leq  C_7 \left(\ln \frac{1}{r} \right) ^{-\frac{p_1 ^- p_2 ^-}{p_2 ^+ - p_1 ^- +1}}\int_r ^{\sqrt{r}} \left( \ln \frac{1}{r} \right) ^{1+\frac{(p_1 ^+ -1) p_2 ^+}{p_2 ^- - p_1 ^+ +1}} \left(\frac{1}{t}\right) ^{\gamma} t^{n-d -1}\textnormal{d}t\\
& =  C_7 \left(\ln \frac{1}{r} \right) ^{-\frac{p_1 ^- p_2 ^-}{p_2 ^+ - p_1 ^- +1}}\left( \ln \frac{1}{r} \right) ^{1+\frac{(p_1 ^+ -1) p_2 ^+}{p_2 ^- - p_1 ^+ +1}}  \frac{r^{\frac{n-d -\gamma}{2}}}{n-d-\gamma} \left(1-r^{\frac{n-d -\gamma}{2}} \right),
\end{split}
\end{equation*}
where $C_i=C_i\left(n , \tau , \mu , \gamma , \varepsilon _3 , p_1 ,  p_2  ,  q _1 ,  q_2 ,  \mathcal{U}\right)>0$, $i=6 ,7$. Since $\lim _{r\rightarrow 0^+} 1/r $ $= \infty$, we can assume that $\left( \ln 1/r \right) ^{1+\frac{(p_1 ^+ -1) p_2 ^+}{p_2 ^- - p_1 ^+ +1}} \leq r^{-\frac{n-d -\gamma}{4}}$. Therefore, if $r\rightarrow 0^+$,
$$
\int_{\Omega _{\delta} }   \frac{\vert    \nabla u\vert     ^{p_1}}{u} \dx+  \int _{\Omega _{\delta}}    \left( \ln \frac{u}{\Lambda  (\delta)} \right)  u^{p_2} \dx = 0.
$$
Hence, $u(x)= \Lambda (\delta)$ almost every in $\Omega_{\delta}$. Thus, we have a contradiction, and proves the Lemma \ref{29}. 
\end{proof}

Now we are ready to prove the main Theorem \ref{34}.

\renewcommand*{\proofname}{\textnormal{\textbf{Proof of Theorem \ref{34}}}}

\begin{proof}
1. First we prove $u \in W^{1, p_1 (\cdot)}(\Omega) \cap L^{\infty}(\Omega)$. As consequence of a Lemma \ref{29}, $u \in L^{\infty}(\Omega)$. Next, for $r<2r_0 / 5 $, let $\psi_{r}: \mathbb{R}\rightarrow \mathbb{R}$ be a smooth function such that 
$$
\psi_{r}(t)=\left\{\begin{aligned} & 0 & & \text { if }  t< r / 2  \text { or } t> 5r / 2  , \\ 
& 1 & & \text { if }  r  < t < 2 r,\end{aligned}\right.
$$
$0 \leq \psi_{r} \leq 1$ and $\left\vert    \psi_{r}^{\prime}\right\vert     \leq C / r$, where $C$ is a suitable positive constant. Set
$$
\varphi=\left(\psi_{r}^{p_{1}^{+}} \circ \dist (\cdot,  \Gamma)\right) u.
$$
We have $\varphi\in W^{1 , p_1 (\cdot)}_{\loc} \left( \bar{\Omega}\backslash \Gamma\right) \cap L^\infty _{\loc}\left( \bar{\Omega}\backslash \Gamma\right)$, with  $\spt \varphi \subset \bar{\Omega}\cap \left\{ r / 2 \leq \dist (\cdot , \Gamma) \leq 5r / 2  \right\}$. For simplicity we write $\psi_{r}=\psi_{r} \circ \dist (\cdot,  \Gamma)$ and $\psi_{r}^{\prime}= \psi_{r}^{\prime} \circ \dist (\cdot ,  \Gamma)$. Substituting $\varphi$ into \eqref{2},
$$
\begin{aligned}
&\int_{\Omega} \left\langle A(\cdot, u, \nabla u), p_{1}^{+} \psi_{r}^{p_{1}^{+}-1} \psi_{r}^{\prime} u \nabla \dist (\cdot , \Gamma) +\psi_{r}^{p_{1}^{+}} \nabla u \right\rangle  \\
& + a(\cdot , u) \psi_{r}^{p_{1}^{+}} u +g(\cdot, u) \psi_{r}^{p_{1}^{+}} u \dx +\int_{\partial \Omega} b(\cdot, u) \psi_{1}^{p_{1}^{+}} u+h(\cdot, u) \psi_{1}^{p_{1}^{+}} u\ds\leq 0
\end{aligned}
$$
By conditions \eqref{4} - \eqref{9},  we have 
$$
\begin{aligned}
&\int_{\Omega}  \mu \vert    \nabla u\vert    ^{p_{1}}  \psi_{r}^{p_{1}^{+}} \dx + \int_{\Omega} \mu \psi_{r}^{p_{1}^{+}}\vert    u\vert    ^{p_{2}+1} \dx +\int_{\partial \Omega} \mu \psi_{r}^{p_{1}^{+}}\vert    u\vert    ^{q_{2}+1} \ds\\
&\leq  \int_{\Omega} \mu^{-1} p_{1}^{+} \psi_{r}^{p_{1}^{+}-1}\left\vert    \psi_{r}^{\prime}\right\vert    \vert    u\vert    \left(\vert    \nabla u\vert    ^{p_{1}-1}+\vert    u\vert    ^{p_{1}-1}+1\right) \dx\\
&+\int_{\Omega} \mu^{-1}\left(\vert    u\vert    ^{p_{1}-1}+1\right) \psi_{r}^{p_{1}^{+}} \vert    u\vert     + \mu^{-1} \psi_{r}^{p_{1}^{+}}\vert    u\vert     \dx\\
&+\int_{\partial \Omega} \mu^{-1}\left(\vert    u\vert    ^{q_{1}-1}+1\right) \psi_{r}^{p_{1}^{+}} \vert    u\vert    +\mu^{-1} \psi_{r}^{p_{1}^{+}}\vert    u\vert     \ds.
\end{aligned}
$$
Since $\vert     u \vert    ^{p_{1}} \leq \max\left\{\Vert      u\Vert      _{L^\infty (\Omega)}^{p_{1}^{+}},\Vert      u\Vert      _{L^{\infty} (\Omega)}^{p_{1} ^-}\right\}$ and $\vert     u \vert    ^{q_{1}} \leq \max\left\{\Vert      u\Vert      _{L^\infty (\partial \Omega)}^{q_{1}^{+}},\Vert      u\Vert      _{L^{\infty} (\partial \Omega)}^{q_{1} ^-}\right\}$,
$$
\begin{aligned}
&	\int_{\Omega}\vert    \nabla u\vert    ^{p_{1}} \psi_{r}^{p_{1}^{+}} \dx \\
&	\leq  \mu^{-2} p_{1}^{+} \int_{\Omega} C_{1}\left(\varepsilon _1  , p_1 \right)\left[\left\vert    \psi _r ^\prime \right\vert    \vert    u\vert     \psi_{r}^{p_{1}^{+}-1-\frac{p^{+} _1(p_1-1)}{p_{1}}}\right]^{p_{1}} 
	+\varepsilon_{1}\left[\vert    \nabla u\vert    ^{p_{1}-1} \psi_{r} ^{\frac{p_{1}^{+}(p_1-1)}{p_{1}}}\right]^{\frac{p_{1}}{p_{1}-1}} \dx 	+C_{2} r^{n-d-1},
\end{aligned}
$$
where $C_{2}$ is a positive constant independent of $r$. Choosing $\varepsilon_{1}=\frac{\mu^{2}}{2 p_{1}^{+}}$,
$$
\int_{\Omega \cap \left\{ r \leq \dist ( \cdot , \Gamma) \leq 2 r\right\}} \vert    \nabla u\vert    ^{p_{1}} \dx \leq C_{3} r^{n-d-p_{1}^{+}},
$$
where $C_3$ is a positive constant independent of $r$. Therefore,
$$
\begin{aligned}
&	\int  _{\Omega \cap \left\{  \dist ( \cdot , \Gamma) \leq 2 r\right\}} \vert    \nabla u\vert    ^{p_{1}} \dx  \\
&	= \sum_{i=0}^{\infty} \int_{\Omega \cap \left\{2^{-i} r \leq \dist ( \cdot , \Gamma) \leq 2^{-i +1} r\right\}}  \vert    \nabla u\vert    ^{p_{1}} \dx \leq C_3 \sum _{i=0}^{\infty} \left(\frac{r}{2^i}\right)^{n-d-p_1 ^+}<\infty.
\end{aligned}
$$
So $\vert    \nabla u\vert    \in L^{p_1(\cdot)} (\Omega)$, and thus we have proved that $u\in W^{1,p_1(\cdot)} (\Omega) \cap L^\infty (\Omega)$.\\

2. Now, we will show that $u$ is a solution of equation \eqref {1} in the domain $\Omega$. For $r\in (0,r_0)$, let $\xi _r:\mathbb{R} \rightarrow \mathbb{R}$ be a smooth function such that
$$
\xi _r (t)= \left\{\begin{aligned}
& 1  & & \text { if } \vert    t\vert     \leq r ,\\ 
& 0 & & \text { if } 2 r \leq \vert    t\vert    ,
\end{aligned}
\right.
$$
$0 \leq \xi \leq 1$ and $\left\vert    \xi_{r}^{\prime}\right\vert     \leq C / r$, where $C$ is a suitable  positive constant. For simplicity we write $\xi _{r}=\xi _{r} \circ \dist (\cdot,  \Gamma)$ and $\xi _{r}^{\prime}= \xi _{r}^{\prime} \circ \dist (\cdot ,  \Gamma)$. Let $\varphi \in W^{1, p_1 (\cdot)}(\Omega) \cap L^{\infty}(\Omega)$. We have  $\left(1-\xi_{r}\circ \dist (\cdot ,  \Gamma) \right) \varphi \in W_{\loc}^{1, p_1(\cdot)}\left(\bar{\Omega} \backslash \Gamma \right) \cap L_{\loc}^{\infty}\left(\bar{\Omega} \backslash \Gamma \right)$, with  $\spt \varphi \subset \bar{\Omega} \backslash \Gamma$. Then, \eqref {2} yields
\begin{equation} \label{40}
\begin{split}
&	\int_{\Omega}  \left\langle A(\cdot , u , \nabla u) ,\left(1-\xi_{r}\right) \nabla \varphi - \varphi  \xi_{r} ^\prime \nabla \dist (\cdot , \Gamma)  \right\rangle + a( \cdot , u ) \left(1-\xi_{r}\right) \varphi 	+  g(\cdot, u) \left(1-\xi_{r}\right) \varphi  \dx \\
&	+ \int _{\partial \Omega } b(\cdot , u) \left(1-\xi_{r}\right) \varphi + h(\cdot , u) \left(1-\xi_{r}\right) \varphi \ds = 0.
\end{split}
\end{equation}
When $r\rightarrow 0^+$, for all $\varphi \in W^{1, p_1 (\cdot)}(\Omega) \cap L^{\infty}(\Omega)$, the equality \eqref{40} implies \eqref{2}. Indeed, we have 
\begin{equation*} 
\begin{split}
&	\lim _{r \rightarrow 0}  \int_{\Omega}  \left\langle A(\cdot , u , \nabla u) ,\left(1-\xi_{r}\right) \nabla \varphi  \right\rangle +  a( \cdot , u ) \left(1-\xi_{r}\right) \varphi  +  g(\cdot, u) \left(1-\xi_{r}\right) \varphi  \dx \\
&	+  \int _{\partial \Omega } b(\cdot , u) \left(1-\xi_{r}\right) \varphi + h(\cdot , u) \left(1-\xi_{r}\right) \varphi \ds \\
&	=\int_{\Omega} \left\langle A(\cdot , u , \nabla u) , \nabla \varphi  \right\rangle  +  a( \cdot , u ) \varphi  +  g(\cdot, u) \varphi  \dx  + \int _{\partial \Omega } b(\cdot , u)  \varphi + h(\cdot , u)  \varphi \ds .
\end{split}
\end{equation*}
Additionally, by \eqref{5}, Proposition \ref{31} and \ref{32}:
$$
\begin{aligned}
&	\left\vert    \int_{\Omega}    \left\langle A( \cdot , u, \nabla u) ,  \varphi  \xi_{r} ^\prime \nabla \dist (\cdot , \Gamma) \right\rangle \dx\right\vert     \\
&	\leq  \frac{C_4 }{r} \int_{\Omega \cap \left\{ r\leq \dist ( \cdot ,\Gamma ) \leq 2r\right\}} \left(\vert    \nabla u\vert    ^{p_1 -1} + \vert    u\vert    ^{p_1 -1} +1 \right)\varphi \dx\\
&	\leq  \frac{C_5}{r} \int_{\Omega \cap \left\{ r\leq \dist ( \cdot ,\Gamma ) \leq 2r\right\}} \vert    \nabla u\vert    ^{p_1 -1} + \vert    u\vert    ^{p_1 -1} +1 \dx\\
&	\leq  \frac{C_6}{r}\left\Vert      \vert    \nabla u\vert    ^{p_1 -1}\right\Vert      _{L^{ \frac{p_1}{p_1-1}}\left(\Omega \cap\left\{ r\leq \dist ( \cdot ,\Gamma ) \leq 2r\right\}\right)}\Vert      1\Vert      _{L^{p_1}\left(\Omega \cap \left\{ r\leq \dist ( \cdot ,\Gamma ) \leq 2r\right\}\right)}
	+C_6 r^{n-d-1}\\
&	\leq \frac{C_7}{r} \max \left\{\left(\int_{\Omega \cap \left\{ r\leq \dist ( \cdot ,\Gamma ) \leq 2r\right\}}\vert    \nabla u\vert    ^{p_1} \dx\right)^{ \left(\frac{p_1 -1 }{p_1} \right) ^+}  \right. 
	, \left.   \left(\int_{\Omega \cap \left\{ r\leq \dist ( \cdot ,\Gamma ) \leq 2r\right\}}\vert    \nabla u\vert    ^{p_1} \dx\right)^{\left( \frac{p_1 -1 }{p_1} \right) ^- }\right\}\\
&	\cdot \left\vert  \Omega \cap  \left\{ r\leq \dist ( \cdot ,\Gamma ) \leq 2r\right\} \right\vert    ^{\frac{1}{p_1^+}}+ C_6 r^{n-d-1}\\
&	\leq  C_8 r^{\frac{n -d - p^+ _1}{p_1 ^+}} \rightarrow 0 \quad \text { as }  r \rightarrow 0,
\end{aligned}
$$
where  $C_i$, $i=4, \ldots, 8$, are positive constants independents of $r$. So, we have obtained that equality \eqref{2} is fulfilled for all $\varphi \in W^{1, p_1(\cdot)}(\Omega) \cap L^{\infty}(\Omega)$. Therefore, the singular set $\Gamma$ is removable for solutions of equation \eqref{1}.

\end{proof}

Similarly as in the Lemma \ref{29} and in the Theorem \ref{34} we have the following results.

\begin{lemma} 
Suppose that the conditions \eqref{4} - \eqref{7}, \eqref{24} - \eqref{51}  and \eqref{50} are satisfied. If $u \in W^{1, p_1(\cdot)} _{\loc} \left(\bar{\Omega}\backslash \Gamma \right)$ $\cap L^{\infty} _{\loc} \left(\bar{\Omega}\backslash \Gamma \right)$ satisfies \eqref{35}, then
$$
\max \{u , 0\} \in L^{\infty} (\Omega).
$$
\end{lemma}

\begin{theorem} 
Suppose that the conditions \eqref{4} - \eqref{7}, \eqref{23} -  \eqref{51}  and \eqref{50} are satisfied. Let $u \in W^{1, p_1(\cdot)} _{\loc} \left(\bar{\Omega}\backslash \Gamma \right)$ $\cap L^\infty _{\loc}\left(\bar{\Omega}\backslash \Gamma\right)$ be a solution of equation \eqref{1} in $\bar{\Omega} \backslash \Gamma$, with $b\equiv h \equiv 0$. Then, the singularity of $u$ at  $\Gamma$ is removable.
\end{theorem}

{\footnotesize
\noindent \textbf{Acknowledgments:} I would like to thank Professor Sergio Almaraz for useful suggestions.
}


\begin{flushright}
\today
\end{flushright}

\end{document}